\documentclass[a4paper, reqno, twoside, 12pt, dvipsnames]{amsart}
\usepackage{amssymb} 
\usepackage{mathtools,thmtools}
\usepackage[all]{xy}
\usepackage{enumerate}
\usepackage[inline]{enumitem}
\usepackage[latin1]{inputenc}
\usepackage{a4wide}
\usepackage[bookmarks=false,colorlinks,urlcolor=cyan,citecolor=red,linkcolor=blue]{hyperref}  
\usepackage{cleveref}
\usepackage{stmaryrd}
\usepackage{lmodern}
\usepackage{microtype}
\usepackage{xcolor}

\makeatletter
\patchcmd{\@setaddresses}{\indent}{\noindent}{}{}
\patchcmd{\@setaddresses}{\indent}{\noindent}{}{}
\patchcmd{\@setaddresses}{\indent}{\noindent}{}{}
\patchcmd{\@setaddresses}{\indent}{\noindent}{}{}
\makeatother

\makeatletter																																													
\DeclareMathSizes{12}{12}{7}{5}																																				
\makeatother																																													

\newtheorem{theorem}{Theorem}[section]
\newtheorem{proposition}[theorem]{Proposition}

\newtheorem*{theorem*}{Theorem}
\newtheorem*{proposition*}{Proposition}
\newtheorem{alphthm}{Theorem}

\theoremstyle{definition}

\newtheorem{example}[theorem]{Example}

\theoremstyle{remark}
\declaretheorem[name=Remark,qed={$\blacktriangle$},sibling=theorem]{remark}

\newtheoremstyle{questions}{}{}{\color{red}}{}{\color{blue}\bfseries}{}{ }{}
\theoremstyle{questions}

\newcommand{\tensor}[1]{\otimes_{#1}} 
\DeclareMathOperator{\cotensor}{\square} 

\newcommand{\op}[1]{#1^{\mathrm{o}}} 

\newcommand{\Ae}{A^{\mathrm{e}}}
\newcommand{\Ao}{A^{\mathrm{o}}}

\newcommand{\cC}{{\mathcal C}}

\newcommand{\cH}{{\mathcal H}}

\newcommand{\cO}{{\mathcal O}}

\newcommand{\cU}{{\mathcal U}}
\newcommand{\cV}{{\mathcal V}}

\newcommand{\Sscript}[1]{#1} 

\newcommand{\ls}[1]{{}_{\Sscript{s}} {#1}}
\newcommand{\lt}[1]{{}_{\Sscript{t}} {#1}}
\newcommand{\rs}[1]{{#1}{}_{\Sscript{s}} }
\newcommand{\rt}[1]{{#1}{}_{\Sscript{t}}}
\newcommand{\tM}[1]{\lt{#1}}
\newcommand{\Mt}[1]{\rt{#1}}
\newcommand{\Mto}[1]{{#1}{}_{\Sscript{\op{t}}}}
\newcommand{\sM}[1]{\ls{{#1}}}
\newcommand{\sMt}[1]{{}_{\Sscript{s}} {#1}{}_{\Sscript{t}}}
\newcommand{\tMs}[1]{{}_{\Sscript{t}} {#1}{}_{\Sscript{s}}}
\newcommand{\sMs}[1]{{}_{\Sscript{s}} {#1}{}_{\Sscript{s}}}
\newcommand{\tMt}[1]{{}_{\Sscript{t}} {#1}{}_{\Sscript{t}}}
\newcommand{\sMto}[1]{{}_{\Sscript{s}} {#1}{}_{\Sscript{\op{t}}}}
\newcommand{\soMt}[1]{{}_{\Sscript{\op{s}}} {#1}{}_{\Sscript{t}}}
\newcommand{\tMso}[1]{{}_{\Sscript{t}} {#1}{}_{\Sscript{\op{s}}}}
\newcommand{\toMs}[1]{{}_{\Sscript{\op{t}}} {#1}{}_{\Sscript{s}}}

\newcommand{\ZZ}{\mathbb{Z}}

\newcommand{\C}{\mathbb{C}}
\newcommand{\K}{\Bbbk}

\newcommand{\coinv}[2]{\prescript{{\sf co}{#2}}{}{#1}} 

\newcommand{\ot}{\otimes}
\newcommand{\tak}[1]{\times_{\Sscript{#1}}}

\newcommand{\img}{\operatorname{Im}}

\newcommand{\calg}{\mathrm{CAlg}}

\newdir{ >}{{}*!/-10pt/@{>}}

\begin{document}
\allowdisplaybreaks

\title[Galois-type correspondence for Hopf algebroids]{A remark on the Galois-type correspondence between ideal coideals and comodule subrings of a Hopf algebroid}

\author{P.~Saracco}
\address{D\'epartement de Math\'ematique, Universit\'e Libre de Bruxelles, Boulevard du Triomphe, B-1050 Brussels, Belgium.}
\email{paolo.saracco@ulb.be}
\urladdr{\url{https://sites.google.com/view/paolo-saracco}}
\urladdr{\url{https://paolo.saracco.web.ulb.be}}

\thanks{PS is a Charg\'e de Recherches of the Fonds de la Recherche Scientifique - FNRS and a member of the National Group for Algebraic and Geometric Structures and their Applications (GNSAGA-INdAM). He would like to thank Joost Vercruysse for discussing the topic together}

\begin{abstract}
We exhibit a bijective correspondence between certain left ideal coideals in a Hopf algebroid for which the resulting quotient is a coequalizer and certain right coideal subrings which are themselves an equalizer, remarkably improving a recent result of the author obtained in collaboration with L.~El Kaoutit, A.~Ghobadi and J.~Vercruysse. Interpreting this result in the Hopf algebra setting provides a bijective correspondence which extends the previously known cases.
\end{abstract}

\pagestyle{headings}

\maketitle


\section*{Introduction}

The study of Galois-type correspondences between ideal coideals and coideal subalgebras in the theory of Hopf algebras has a long tradition of outstanding results, starting from the renowned paper \cite{Takeuchi:1972} of Takeuchi, where he establishes a bijective correspondence between sub-Hopf algebras of a commutative Hopf algebra $H$ and normal Hopf ideals in $H$. From this, one might deduce that the category of commutative and cocommutative Hopf algebras over a field (and hence that of affine abelian group schemes) is an abelian category: see \cite[Corollary 4.16 and Theorem 5.4]{Takeuchi:1972}. A cocommutative analogue of Takeuchi's result, establishing a one-to-one correspondence between Hopf subalgebras and left ideal coideals in a cocommutative Hopf algebra, was exhibited by Newman in \cite{Newman} and it played a central role in showing that the category of cocommutative Hopf algebras over a field is semi-abelian (see \cite{GSV}). More cases of interest from the literature are reviewed in \Cref{sec:HopfAlgs}.

Very recently, the interest in this correspondence has been extended to the Hopf algebroid setting in \cite{AryanLaiachiPaoloJoost}, with the aim of discussing a quotient theory of affine groupoid schemes from a Hopf algebraic perspective. A deeper analysis of the results therein led to realise that a necessary condition for a left ideal coideal $I$ in a Hopf algebroid $\cH$ to be of the form $\cH B^+$ for a certain coideal subring $B$ is that the quotient $\cH/I$ is the coequalizer of the pair $\xymatrix @C=70pt{\cH \cotensor^{\cH/I} \cH \ar@<+0.55ex>[r]|(0.7){\,\varepsilon \tensor{A} \cH\,} \ar@<-0.55ex>[r]|(0.42){\,\cH \tensor{A} \varepsilon\,} & \cH}$. Similarly, a necessary condition for a coideal subring $B$ to be of the form $\coinv{\cH}{\frac{\cH}{I}}$ for a certain left ideal coideal $I$ is that $B$ is the equalizer of the pair $\xymatrix @C=70pt{\cH \ar@<+0.55ex>[r]|(0.55){\,\iota \tensor{B} \cH\,} \ar@<-0.55ex>[r]|(0.27){\,\cH \tensor{B} \iota\,} & \cH \tensor{B} \cH}$. These conditions can be made sufficient if we ask certain maps to be injective. Namely, we are going to prove the following central result (refer to \Cref{ssec:purity} for the notion of purity used therein).

\begin{alphthm}[\Cref{thm:main} below]
Let $\cH$ be a left Hopf algebroid over $A$ such that $\sM{\cH} = {}_{{A\tensor{}\op{1}}}\cH$ is $A$-flat.
Then we have an inclusion-preserving bijective correspondence
\[
\xymatrix @R=0pt{
{\left\{ \begin{array}{c} \text{left ideal coideals } I \text{ for which} \\ \cH \cotensor^{\frac{\cH}{I}} \cH \rightrightarrows \cH \to \cH/I \text{ is a} \\ \text{coequalizer and } \coinv{\cH}{\frac{\cH}{I}} \subseteq \cH \text{ is} \\  \text{left and right }\cH\text{-pure} \\ \text{as } \op{A}\text{-bimodule} \end{array} \right\}} \ar@<+0.5ex>[r]^-{\Psi} & {\left\{ \begin{array}{c} \text{right } \cH\text{-comodule } \op{A}\text{-subrings } B \\ \text{via }t \text{ for which } B \to \cH \rightrightarrows \cH \tensor{B} \cH \\ \text{is an equalizer, } B \subseteq \cH \text{ is left}  \\ \text{and right }\cH\text{-pure as } \op{A}\text{-bimodule and} \\ B \text{ is compatible with the translation map}\end{array} \right\}} \ar@<+0.5ex>[l]^-{\Phi} \\
I \ar@{|->}[r] & \coinv{\cH}{\frac{\cH}{I}} \\
\cH B^+ & B. \ar@{|->}[l]
}
\]
\end{alphthm}

In spite of the technical requirements that appear in the Hopf algebroid setting, the very same theorem rephrased for Hopf algebras over a field reads as follows.

\begin{alphthm}[\Cref{thm:Hopfcase} below]
Let $H$ be a Hopf algebra over the field $\K$.
Then, the following is an inclusion-preserving bijective correspondence:
\[
\xymatrix @R=0pt{
{\left\{ \begin{array}{c} \text{left ideal coideals} \\ \text{in } H \text{ for which} \\ H \cotensor^{\frac{H}{I}} H \rightrightarrows H \to H/I \\ \text{is a coequalizer} \end{array} \right\}} \ar@<+0.5ex>[r]^-{\Psi} & {\left\{ \begin{array}{c} \text{right coideal subalgebras} \\ \text{of } H \text{ such that} \\ B \to H \rightrightarrows H \tensor{B} H \\ \text{is an equalizer}\end{array} \right\}} \ar@<+0.5ex>[l]^-{\Phi} \\
I \ar@{|->}[r] & \coinv{H}{\frac{H}{I}} \\
H B^+ & B. \ar@{|->}[l]
}
\]
\end{alphthm}

The latter bijection turns out to extend the already known ones from the literature, which can be seen as restrictions of this new one, as we discuss in \Cref{sec:HopfAlgs}.
Concretely, after collecting the required preliminaries in \Cref{sec:prelim}, we will look again at the correspondence between coideal subrings and left ideal coideals in bialgebroids and we will establish the aforementioned necessary conditions in \Cref{sec:bialgd}. Then, we will prove our main \Cref{thm:main} in \Cref{sec:Hoids} and we will specialise it to the Hopf algebra setting in \Cref{sec:HopfAlgs}. Examples are provided that motivate the importance of these results, even in the classical framework (see \Cref{ex:extending}), and that show how the novel, milder, conditions cannot be avoided as they are not automatically satisfied in general (see \Cref{rem:extending}).


\section{Preliminaries}\label{sec:prelim}

Let us fix a field $\K$, of arbitrary characteristic and not necessarily algebraically closed (unless stated otherwise). All objects, such as algebras and modules over them, are supposed to have an underlying vector space structure over $\K$. The unadorned tensor product $\tensor{}$ stands for $\tensor{\K}$. Furthermore, we assume a certain familiarity of the reader with the language of monoidal categories and of (co)monoids therein (see, for example, \cite[VII]{MacLane}).

Given a (preferably, non-commutative) $\K$-algebra $A$, we denote by $\Ao$ its opposite algebra and by $\Ae$ its enveloping algebra $A \tensor{} \op{A}$. In this case, an element $a \in A$ may be denoted by $\op{a}$ when it is helpful to stress that it is viewed as an element in $\op{A}$.

As a matter of notation, we often denote the unit and multiplication of an algebra or a ring by $u$ and $\mu$, respectively, and the counit and comultiplication of a coalgebra or coring by $\varepsilon$ and $\Delta$, respectively. In addition, we take advantage of the notation explained in \cite[\S2.1]{AryanLaiachiPaoloJoost}: given an $\Ae$-bimodule $M$, the $A$-action by elements of the form $a\tensor{}\op{1}$ is denoted by $\ls{M}$ or $\rs{M}$ and the $\op{A}$-action of the element of the form $1\tensor{}\op{a}$ by $\lt{M}$ or $\rt{M}$, where $s\colon  A \to \Ae, a \mapsto a \ot \op{1}$ and $t\colon \Ao \to \Ae, \op{a} \mapsto 1 \ot \op{a}$. If instead we need to consider actions by the elements $\op{t(a)}$ or $\op{s(a)}$ to produce other $A$ or $\op{A}$ actions, then we use the notation $\op{t}$ and $\op{s}$ on the side on which we are declaring the action. Summing up:
\begin{gather*}
\ls{M}  \coloneqq  {}_{{A\tensor{}\op{1}}}M, 
\quad 
\lt{M}  \coloneqq  {}_{{1\tensor{}\op{A}}}M, 
\quad 
\rt{M}  \coloneqq  {M}_{{1\tensor{}\op{A}}}, 
\quad  
\rs{M}  \coloneqq  {M}_{{A\tensor{}\op{1}}}, 
\\ 
\sMs{M} \coloneqq  {}_{{A\tensor{}\op{1}}}{M}_{{A\tensor{}\op{1}}}, 
\quad 
\sMt{M} \coloneqq  {}_{{A\tensor{}\op{1}}}{M}_{{1\tensor{}\op{A}}}, 
\quad 
\tMs{M} \coloneqq  {}_{{1\tensor{}\op{A}}}{M}_{{A\tensor{}\op{1}}}, 
\quad   
\tMt{M} \coloneqq  {}_{{1\tensor{}\op{A}}}{M}_{{1\tensor{}\op{A}}}, \\
\underset{A\text{-bimodule}}{\sMto{M} 			\coloneqq		{}_{{A\tensor{}\op{A}}}{M}}, 
\quad 
\underset{\op{A}\text{-bimodule}}{\tMso{M} 	\coloneqq 	{}_{{A\tensor{}\op{A}}}{M}}, 
\quad  
\underset{\op{A}\text{-bimodule}}{\soMt{M} 	\coloneqq		{M}_{{A\tensor{}\op{A}}}}, 
\quad
\underset{A\text{-bimodule}}{\toMs{M} 			\coloneqq		{M}_{{A\tensor{}\op{A}}}},
\end{gather*}


\subsection{Cotensor product}

Let $A$ be a $\K$-algebra and $\cC$ be an $A$-coring. If \(M\) is a right \(\cC\)-comodule and \(N\) is a left $\cC$-comodule, then we can define their \emph{cotensor product} to be the equalizer
\[
\xymatrix @C=18pt{
M \cotensor^\cC N  \ar[r] & M \tensor{A} N \ar@<+0.5ex>[rr]^-{\delta_M \tensor{A} N} \ar@<-0.5ex>[rr]_-{M \tensor{A} \delta_N} & & M \tensor{A} \cC \tensor{A} N
}
\]
in the category of vector spaces.

\subsection{Purity}\label{ssec:purity}

Let $R$ be a ring. Recall that a morphism of left (resp. right) $R$-modules $f \colon M \to N$ is called \emph{pure} (or \emph{universally injective}) if the morphism of abelian groups $P \tensor{R} f \colon P \tensor{R} M \to P \tensor{R} N$ (resp. the morphism $f \tensor{R} P \colon M \tensor{R} P \to N \tensor{R} P$) is injective for every right (resp. left) $R$-module $P$. 

As a matter of terminology, given an injective morphism of left (resp. right) $R$-modules $f \colon M \to N$ and a right (resp. left) $R$-module $P$, we will say that $f$ is \emph{$P$-pure} if $P \tensor{R} f \colon P \tensor{R} M \to P \tensor{R} N$ (resp. $f \tensor{R} P \colon M \tensor{R} P \to N \tensor{R} P$) is injective.

In particular, a pure morphism is $P$-pure for every $P$.


\subsection{The Takeuchi-Sweedler crossed product}
Let $A$ be a $\K$-algebra and ${}_{\Ae}M{}_{\Ae}$ and ${}_{\Ae}N{}_{\Ae}$ be two $\Ae$-bimodules. We first define the $A$-bimodule 
\[
M \tensor{A} N \coloneqq  \lt{M} \tensor{A}  \ls{N} = \frac{M \ot N}{\Big\langle t(\op{a})m \ot n - m \ot s(a)n ~\big\vert~ m\in M, n\in N, a\in A \Big\rangle}.
\]
This is the tensor product over $A$ that we consider more often, unless specified otherwise.
Then, inside $M \tensor{A} N$ we consider the subspace
\[
M \tak{A} N \coloneqq \left\{ \left. \sum_i m_i \tensor{A} n_i \in M\tensor{A}N \ \right| \ \sum_i m_it(\op{a}) \tensor{A} n_i = \sum_i m_i \tensor{A} n_is(a)\right\},
\]
which we call the \emph{Takeuchi-Sweedler crossed product}.
If $\cU$ and $\cV$ are two $\Ae$-rings (i.e., $\K$-algebras with a $\K$-algebra map $\Ae \to \cU$), then $\cU \tak{A} \cV$ is also an $\Ae$-ring with multiplication 
\[\left(\sum_iu_i \tensor{A} v_i\right)\cdot \left(\sum_ju'_j \tensor{A} v'_j\right) = \sum_{i,j}u_iu'_j\tensor{A}v_iv'_j\]
for all $u_i,u'_j\in \cU$, $v_i,v'_j\in \cV$ and $\K$-algebra morphism $\Ae \to \cU \tak{A} \cV, a \ot \op{b} \mapsto s_\cU(a) \tensor{A} t_\cV\big(\op{b}\big)$.


\subsection{Left bialgebroids}\label{ssec:n1}

Fix a base algebra $A$. Recall that a \emph{(left) bialgebroid over $A$} is an $\Ae$-ring $\cH$ with commuting source $s\colon A\to \cH$ and target $t\colon \op{A}\to\cH$, together with an $A$-coring structure $\left(\cH,\Delta,\varepsilon\right)$ on the $A$-bimodule ${}_{\Ae}\cH=\sMto{\cH}$, subject to the following compatibility conditions
\begin{enumerate}[leftmargin=1.2cm,label=(B\arabic*),ref=(B\arabic*)]
\item\label{item:B4} $\Delta$ takes values into $\sMto{\cH} \tak{A} \sMto{\cH}$ and $\Delta\colon  \sMto{\cH} \to \sMto{\cH}\tak{A}\sMto{\cH}$ is a morphism of $\K$-algebras;
\item\label{item:B5} $\varepsilon\Big(xs\big(\varepsilon\left(y\right)\big)\Big) = \varepsilon\left(xy\right) = \varepsilon\Big(xt\big(\op{\varepsilon\left(y\right)}\big)\Big)$ for all $x,y\in\cH$;
\item\label{item:B6} $\varepsilon(1_\cH)=1_A$.
\end{enumerate}
Henceforth, all bialgebroids will be left ones and over $A$, so we often omit to specify it.

As a matter of notation, for $\cH$ a bialgebroid and $B\subseteq \cH$ we set
\[
\cH^+\coloneqq \ker(\varepsilon) \qquad \text{and} \qquad B^+\coloneqq B \cap \cH^+.
\]
Denote by $\iota_B\colon B \to \cH$ the inclusion. Assume that $B$ is an $\op{A}$-subring of $\cH$ via $t$, that is, $B$ is an $\op{A}$-ring via a $\K$-algebra extension $t'\colon \op{A}\to B$ such that $\iota_B \circ t'= t$.
If we consider the restriction $\varepsilon' \coloneqq \varepsilon \circ \iota_B$ of $\varepsilon$ to $B$, then the latter is a right $A$-linear morphism and $B^+ = \cH^+ \cap B = \ker(\varepsilon)\cap B = \ker(\varepsilon')$.

Recall that an $\Ao$-subring $B$ of $\cH$ via $t$ is a \emph{right $\cH$-comodule $\Ao$-subring} if it admits a right $A$-linear coaction $\delta \colon  \Mto{B} \to \Mto{B} \tensor{A} \sMto{\cH}$ such that the following diagram commutes
\[
\xymatrix @R=12pt{B \ar[r]^-{\delta} \ar[d]_-{\iota} & B \tensor{A} \cH \ar[d]^-{\iota \tensor{A} \cH} \\ \cH \ar[r]_-{\Delta} & \cH \tensor{A} \cH. }
\]
A \emph{coideal} $N$ in a bialgebroid $\cH$ is an $A$-subbimodule of $\sMto{\cH}$ such that
\[
\Delta(N) \subseteq \img \Big( N \tensor{A} \cH + \cH \tensor{A} N\Big)  \qquad \text{and} \qquad \varepsilon(N)=0_A,
\] 
where $\img(-)$ denotes the canonical image in the tensor product $\sMto{\cH} \tensor{A} \sMto{\cH}$.

\subsection{Left Hopf algebroids}

Let $\cH$ be a left bialgebroid and consider the tensor product
\[
\cH \tensor{\op{A}} \cH = \Mt{\cH} \tensor{\op{A}} \tM{\cH} \coloneqq \frac{\cH \otimes \cH}{\Big\langle xt(\op{a}) \otimes y - x \otimes t(\op{a})y ~\big\vert~ a \in A, x,y \in \cH \Big\rangle}.
\]
Inside $\cH \tensor{\op{A}} \cH$ we isolate the distinguished subspace
\[
\cH \tak{\op{A}} \cH \coloneqq \left\{\left.\sum_ix_i \tensor{\op{A}} y_i ~\right|~ \sum_it(\op{a})x_i \tensor{\op{A}} y_i = \sum_ix_i \tensor{\op{A}} y_it(\op{a})\right\}.
\]
The latter is a $\K$-algebra with unit $1 \tensor{\op{A}} \op{1}$ and multiplication
\[
\left(\sum_ix_i \tensor{\op{A}} y_i\right)\left(\sum_ju_j \tensor{\op{A}} v_j\right) \coloneqq \sum_{i,j}x_iu_j \tensor{\op{A}} v_jy_i.
\]
We say that $\cH$ is a \emph{left Hopf algebroid} in the sense of \cite[Theorem and Definition 3.5]{Schauenburg} if the canonical map
\begin{equation}\label{Eq:betamap}
\beta\colon  \Mt{\cH} \tensor{\op{A}} \tM{\cH} \to \Mto{\cH} \tensor{A} \sM{\cH}, \qquad x \tensor{\op{A}} y \mapsto \sum x_1 \tensor{A} x_2y,
\end{equation}
is invertible. In such a case, the assignment
\[\cH \to \Mt{\cH} \tensor{\op{A}} \tM{\cH}, \qquad x \mapsto \beta^{-1}(x \tensor{A} 1_\cH) \eqqcolon \sum x_{+} \tensor{\op{A}} x_{-}\]
induces a morphism of $\K$-algebras
\[
\gamma\colon \cH \to \cH \tak{\op{A}} \cH, \qquad x \mapsto \beta^{-1}(x \tensor{A} 1_\cH).
\]
The map $\gamma$ is referred to as the \emph{(left) translation map}.


\section{Left ideal coideals and right coideal subrings in bialgebroids}\label{sec:bialgd}

Let us begin by summarizing the content of \cite[\S3.1]{AryanLaiachiPaoloJoost}.
Let $\cH$ be a left bialgebroid.
If $\sM{\cH} = {}_{{A\tensor{}\op{1}}}\cH$ is $A$-flat, then we have well-defined inclusion-preserving assignments
\begin{equation}\label{eq:PhiPsi}
\begin{gathered}
\xymatrix @R=0pt{
{\left\{ \begin{array}{c} \text{left ideals} \\ \text{coideals in } \cH \end{array} \right\}} \ar@<+0.5ex>[r]^-{\Psi} & {\left\{ \begin{array}{c} \text{right } \cH\text{-comodule} \\ \op{A}\text{-subrings via }t \text{ in } \cH \end{array} \right\}} \ar@<+0.5ex>[l]^-{\Phi} \\
I \ar@{|->}[r] & \coinv{\cH}{\frac{\cH}{I}} \\
\cH B^+ & B \ar@{|->}[l]
}
\end{gathered}
\end{equation}
where
\[\coinv{\cH}{\frac{\cH}{I}} = \left\{x \in \cH ~\left|~ \sum (x_1 + I) \tensor{A} x_2 = (1+I) \tensor{A} x \right.\right\}.\]
We also have canonical inclusions $B \subseteq \coinv{\cH}{\frac{\cH}{\cH B^+}}$, which we denote by $\eta_B$, and $\cH B^+ \subseteq I$, that we denote by $\epsilon_I$. Furthermore,
\begin{itemize}[leftmargin=0.7cm]
\item If $I$ is a left ideal coideal in $\cH$ and $B \coloneqq \coinv{\cH}{\frac{\cH}{I}}$, then $B = \coinv{\cH}{\frac{\cH}{\cH B^+}}$, i.e.\ $\Psi\Phi\Psi(I) = \Psi(I)$.
\item If $B$ is a right $\cH$-comodule $\Ao$-subring of $\cH$ via $t$ and $I \coloneqq \cH B^+$, then $I = \cH \left(\coinv{\cH}{\frac{\cH}{I}}\right)^+$. That is, $\Phi\Psi\Phi(B) = \Phi(B)$.
\end{itemize}
In other words, $\Phi$ and $\Psi$ form a \emph{monotone Galois connection} (or, equivalently, an adjunction) between the two lattices. We are interested in determining under which circumstances this Galois-type correspondence induces a bijection.

\begin{proposition}\label{prop:coidealsubalgebra}
Let $B$ be a right $\cH$-comodule $\op{A}$-subring of $\cH$ via $t$ and consider the canonical projection $\pi_B \colon \cH \to \cH/\cH B^+$, which is a morphism of $A$-corings and left $\cH$-modules, and the canonical inclusion $\iota_B \colon B \to \cH$. Then
\begin{enumerate}[leftmargin=0.8cm,label=(\roman*),ref={\itshape(\roman*)}]
\item\label{item:2.1a} The following diagram is commutative, where $\zeta \colon \Mt{\cH} \tensor{\op{A}} {_{t'}B} \to \cH \cotensor^{\frac{\cH}{\cH B^+}} \cH$ is given by $\zeta(x \tensor{\op{A}} b) = \sum x_1 \tensor{A} x_2\iota_B(b)$ for all $x \in \cH$ and $b \in B$, the left-hand square commutes sequentially, and whose rows are coequalizers in the category of vector spaces
\[
\xymatrix @C=20pt @R=15pt{
\cH \tensor{\op{A}} B \ar@<+0.5ex>[rr]^-{\cH \tensor{\op{A}} \varepsilon'} \ar@<-0.5ex>[rr]_-{\mu \circ (\cH \tensor{\op{A}} \iota_B)} \ar[d]_-{\zeta} && \cH \ar[r]^-{\pi_B} \ar@{=}[d] & {\displaystyle \frac{\cH}{\cH B^+} } \ar@{=}[d] \\
\cH \cotensor^{\frac{\cH}{\cH B^+}} \cH \ar@<+0.5ex>[rr]^-{\cH \tensor{A} \varepsilon} \ar@<-0.5ex>[rr]_-{\varepsilon \tensor{A} \cH} && \cH \ar[r]_-{\pi_B} & {\displaystyle \frac{\cH}{\cH B^+} }.
}
\]
\item\label{item:2.1b} The following diagram commutes sequentially, where $\xi \colon \cH_{\iota} \tensor{B} {_{\iota}{\cH}} \to \Mto{\cH/\cH B^+} \tensor{A} \sM{\cH}$ is given by $\xi(x \tensor{B} y) = \sum \pi_B(x_1) \tensor{A} x_2y$ for all $x,y \in \cH$ and whose second row is an equalizer in the category of vector spaces
\begin{equation}\label{eq:PsiPhiB}
\begin{gathered}
\xymatrix @C=25pt @R=18pt{
 & \cH \ar@<+0.5ex>[rr]^-{\cH \tensor{B} \iota_B} \ar@<-0.5ex>[rr]_-{\iota_B \tensor{B} \cH} \ar@{=}[d] && \cH \tensor{B} \cH \ar[d]^-{\xi} \\
\coinv{\cH}{\frac{\cH}{\cH B^+}} \ar[r] & \cH \ar@<+0.5ex>[rr]^-{(\pi_B \tensor{A} \cH) \Delta} \ar@<-0.5ex>[rr]_-{\pi_B t \, \tensor{A} \cH} && {\displaystyle \frac{\cH}{\cH B^+} } \tensor{A} \cH.
}
\end{gathered}
\end{equation}
\end{enumerate}
\end{proposition}

\begin{proof}
Let us begin by proving \ref{item:2.1a}. A straightforward check shall convince the reader that the first line is a coequalizer diagram. Concerning the second line, if $\sum_i x_i \tensor{A} y_i \in \cH \cotensor^{\frac{\cH}{\cH B^+}} \cH$, then 
\[\sum_i {x_i}_1 \tensor{A} \pi_B\big({x_i}_2\big) \tensor{A} y_i = \sum_i {x_i} \tensor{A} \pi_B\big({y_i}_1\big) \tensor{A} {y_i}_2\]
and hence, by applying $\varepsilon \tensor{A} \cH/\cH B^+ \tensor{A} \varepsilon$ to both sides,
\[\pi_B\big(\sum_i {x_i} \varepsilon(y_i)\big) = \pi_B\big(\sum_i \varepsilon({x_i}) {y_i}\big).\]
Furthermore, for every $x \in \cH$ and $b \in B$ we have that $\sum x_1 \tensor{A} x_2 \iota_B(b) \in \cH \cotensor^{\frac{\cH}{\cH B^+}} \cH$ because
\begin{align*}
\sum x_1 \tensor{A} \pi_B\big(x_2 \iota_B(b)_1\big) \tensor{A} x_3 \iota_B(b)_2 & = \sum x_1 \tensor{A} \pi_B\big(x_2 \iota_B(b_0)\big) \tensor{A} x_3 b_1 \\
& = \sum x_1 \tensor{A} \pi_B(x_2) \varepsilon\iota_B(b_0) \tensor{A} x_3 b_1 \\
& = \sum x_1 \tensor{A} \pi_B(x_2) \tensor{A} x_3 \iota_B(b).
\end{align*}
Therefore, $\zeta$ lands indeed in the cotensor product and if $f \colon \cH \to V$ is a linear map such that $f \circ (\cH \tensor{A} \varepsilon) = f \circ (\varepsilon \tensor{A} \cH)$ then for all $x \in \cH$ and $b \in B^+$ we have
\begin{align*}
f(x\iota_B(b)) & = \big(f \circ (\varepsilon \tensor{A} \cH)\big)\big(\sum x_1 \tensor{A} x_2 \iota_B(b)\big) = \big(f \circ (\cH \tensor{A} \varepsilon)\big)\big(\sum x_1 \tensor{A} x_2 \iota_B(b)\big) \\
& = f\big(\sum x_1 \varepsilon\big(x_2 s\varepsilon\iota_B(b)\big)\big) = 0
\end{align*}
and so $f$ factors uniquely through $\cH/\cH B^+$, showing that it is indeed the coequalizer. To finish the check that $\zeta$ is well-defined, notice that it is simply the corestriction of the well-defined composition $\beta \circ (\cH \tensor{\op{A}} \iota_B)$, where $\beta$ is the canonical map \eqref{Eq:betamap} of $\cH$. The commutativity of the diagram can be checked directly.

Concerning \ref{item:2.1b}, the existence of $\xi$ induced by $\beta$ has been already shown in \cite[Lemma 3.12]{AryanLaiachiPaoloJoost}. The key observation is that the following diagram is commutative
\begin{equation}\label{eq:Bxi}
\begin{gathered}
\xymatrix @C=40pt {
\Mt{\cH} \tensor{\op{A}} {_{t'}B_{t'}} \tensor{\op{A}} \tM{\cH} \ar[d]_-{\theta} \ar@<+0.5ex>[rr]^-{\mu(\cH \tensor{\op{A}}\iota_B) \tensor{\op{A}} \cH} \ar@<-0.5ex>[rr]_-{\cH \tensor{\op{A}} \mu(\iota_B \tensor{\op{A}} \cH)} && \Mt{\cH} \tensor{\op{A}} \tM{\cH} \ar[d]^-{\beta} \ar[r]^-{p} & \cH_\iota \tensor{B} {_\iota\cH} \ar@{.>}[d]^-{\xi} \\
\Mto{\big(\Mt{\cH} \tensor{\op{A}} {_{t'}B}\big)} \tensor{A} \sM{\cH} \ar@<+0.5ex>[rr]^-{\mu(\cH \tensor{\op{A}}\iota_B) \tensor{A} \cH} \ar@<-0.5ex>[rr]_-{\cH \tensor{\op{A}} \varepsilon' \tensor{A} \cH} && \Mto{\cH} \tensor{A} \sM{\cH} \ar[r]_-{\pi_B \tensor{A} \cH} & {\displaystyle \Mto{\frac{\cH}{\cH B^+}} \tensor{A} \sM{\cH}}
}
\end{gathered}
\end{equation}
where
\[
\theta \colon \cH \tensor{\op{A}} B \tensor{\op{A}} \cH \to \cH \tensor{\op{A}} B \tensor{A} \cH, \qquad x \tensor{\op{A}} b \tensor{\op{A}} y \mapsto \sum x_1 \tensor{\op{A}} b_0 \tensor{A} x_2b_1y,
\]
and the rows are coequalizers.
The second row is an equalizer by definition of the space of coinvariant elements. The commutativity of the diagram can be checked directly.
\end{proof}

\begin{proposition}\label{prop:idealcoideal}
Let $I$ be a left ideal coideal in $\cH$ and consider the space of coinvariant elements $\coinv{\cH}{\frac{\cH}{I}}$, which is a right $\cH$-comodule $\op{A}$-subring of $\cH$ via $t$ with canonical inclusion $\iota_I \colon \coinv{\cH}{\frac{\cH}{I}} \to \cH$, and the canonical projection $\pi_I \colon \cH \to \cH/I$. For the sake of clarity, we often denote by $\tensor{\coinv{\cH}{}}$ the tensor product over $\coinv{\cH}{\frac{\cH}{I}}$, when there is no risk of confusion. Then:
\begin{enumerate}[leftmargin=0.8cm,label=(\roman*),ref={\itshape(\roman*)}]
\item\label{item:2.2a} The following diagram is commutative, where $\xi \colon \cH \tensor{\coinv{\cH}{\frac{\cH}{I}}} \cH \to \cH/I \tensor{A} \cH$ is given by $\xi(x \tensor{\coinv{\cH}{}} y) = \sum \pi_I(x_1) \tensor{A} x_2y$ for all $x,y \in \cH$, the right-hand side square commutes sequentially, and whose rows are equalizers in the category of vector spaces
\[
\xymatrix @C=22pt {
\coinv{\cH}{\frac{\cH}{I}} \ar[r] \ar@{=}[d] & \cH \ar@{=}[d] \ar@<+0.5ex>[rr]^-{\cH \tensor{\coinv{\cH}{}} \iota_I} \ar@<-0.5ex>[rr]_-{\iota_I \tensor{\coinv{\cH}{}} \cH} && \cH \tensor{\coinv{\cH}{\frac{\cH}{I}}} \cH \ar[d]^-{\xi} \\
\coinv{\cH}{\frac{\cH}{I}} \ar[r] & \cH \ar@<+0.5ex>[rr]^-{(\pi_I \tensor{A} \cH) \circ \Delta} \ar@<-0.5ex>[rr]_-{\pi_I t \, \tensor{A} \cH} && {\displaystyle \frac{\cH}{I} \tensor{A} \cH}.
}
\]
\item\label{item:2.2b} The following diagram commutes sequentially, where $\zeta \colon \Mt{\cH} \tensor{\op{A}} \tM{\coinv{\cH}{\frac{\cH}{I}}} \to \cH \cotensor^{\frac{\cH}{I}} \cH$ is given by $\zeta(x \tensor{\op{A}} b) = \sum x_1 \tensor{A} x_2\iota_I(b)$ for all $x \in \cH$ and $b \in \coinv{\cH}{\frac{\cH}{I}}$, and whose first row is a coequalizer in the category of vector spaces
\begin{equation}\label{eq:PhiPsiI}
\begin{gathered}
\xymatrix @C=20pt {
\cH \tensor{\op{A}} \coinv{\cH}{\frac{\cH}{I}} \ar[d]_-{\zeta} \ar@<+0.5ex>[rr]^-{\mu} \ar@<-0.5ex>[rr]_-{\cH \tensor{\op{A}} \varepsilon} && \cH \ar@{=}[d] \ar[r] & {\displaystyle \frac{\cH}{\cH\big(\coinv{\cH}{\frac{\cH}{I}}\big)^+}} \\
\cH \cotensor^{\frac{\cH}{I}} \cH \ar@<+0.5ex>[rr]^-{\varepsilon \tensor{A} \cH} \ar@<-0.5ex>[rr]_-{\cH \tensor{A} \varepsilon} && \cH &
}
\end{gathered}
\end{equation}
\end{enumerate}
\end{proposition}

\begin{proof}
The fact that the second line in \ref{item:2.2a} is an equalizer follows from the definition of the space of coinvariant elements. To show that the first line is an equalizer as well let us proceed as follows. First of all, notice that if $b \in \coinv{\cH}{\frac{\cH}{I}}$ and $x \in \cH$ then
\[(\pi_I \tensor{A} \cH)\big(\Delta(xb)\big) = \sum \pi_I(x_1 b_1) \tensor{A} x_2b_2 = x\cdot \left(\sum \pi_I(b_1) \tensor{A} b_2\right) = \sum \pi_I(x_1) \tensor{A} x_2b,\]
where $\cdot$ is simply the diagonal action of $\cH$ on the tensor product of two $\cH$-modules, so that $(\pi_I \tensor{A} \cH)\circ \Delta$ is right $\coinv{\cH}{\frac{\cH}{I}}$-linear. Thus, if $x \in \cH$ satisfies $1 \tensor{\coinv{\cH}{}} x = x \tensor{\coinv{\cH}{}} 1$, then we may apply first $\big((\pi_I \tensor{A} \cH)\circ \Delta\big) \tensor{\coinv{\cH}{}} \cH$ and then $\cH/I \tensor{A} \mu$ to both sides to conclude that
\[\sum \pi_I(x_1) \tensor{A} x_2 = \pi_I(1) \tensor{A} x,\]
that is to say, $x \in \coinv{\cH}{\frac{\cH}{I}}$ and so $\coinv{\cH}{\frac{\cH}{I}}$ is indeed the equalizer of $\iota_I \tensor{\coinv{\cH}{}} \cH$ and $\cH \tensor{\coinv{\cH}{}} \iota_I$. We are left to prove that $\xi$ is well-defined, because the commutativity of the diagram is again a straightforward check. For all $x,y \in \cH$ and $b \in \coinv{\cH}{\frac{\cH}{I}}$ we have that
\[\sum \pi_I(x_1b_1) \tensor{A} x_2b_2y = x \cdot \left(\sum \pi_I(b_1) \tensor{A} b_2\right) y = \sum \pi_I(x_1) \tensor{A} x_2by,\]
where $\cdot$ is again the diagonal action of $\cH$ and the juxtaposition on the right is the regular right $\cH$-action, so that $(\pi_I \tensor{A} \cH) \circ \beta$ factors through the tensor product over $\coinv{\cH}{\frac{\cH}{I}}$, as desired, where $\beta$ is the canonical map \eqref{Eq:betamap} of $\cH$.

Concerning \ref{item:2.2b}, the fact that the first row is an equalizer follows, as before, from a direct check, and the fact that the diagram is commutative, too. We are left to prove that the morphism $\beta \circ (\cH \tensor{\op{A}} \iota_I)$ corestricts to the cotensor product, in order to conclude that $\zeta$ is well-defined. This descends from the commutativity of the following diagram
\begin{equation}\label{eq:Izeta}
\begin{gathered}
\xymatrix @C=35pt {
\cH \tensor{\op{A}} \coinv{\cH}{\frac{\cH}{I}} \ar@{.>}[d]_-{\zeta} \ar[r] & \cH \tensor{\op{A}} \cH \ar[d]_-{\beta} \ar@<+0.5ex>[rr]^-{\cH \tensor{\op{A}} \pi_I t \tensor{A} \cH} \ar@<-0.5ex>[rr]_-{\cH \tensor{\op{A}} (\pi_I \tensor{A} \cH) \Delta} && {\displaystyle \cH \tensor{\op{A}} \left(\frac{\cH}{I} \tensor{A} \cH\right)} \ar[d]^-{\beta \tensor{\cH} ({\cH}/{I} \tensor{A} \cH)} \\
\cH \cotensor^{\frac{\cH}{I}} \cH \ar[r] & \cH \tensor{A} \cH \ar@<+0.5ex>[rr]^-{(\cH \tensor{A} \pi_I) \Delta \tensor{A} \cH} \ar@<-0.5ex>[rr]_-{\cH \tensor{A} (\pi_I \tensor{A} \cH) \Delta} && {\displaystyle \cH \tensor{A} \left(\frac{\cH}{I} \tensor{A} \cH\right)}
} 
\end{gathered}
\end{equation}
whose second line is an equalizer in the category of vector spaces.
\end{proof}


\section{Left ideal coideals and right coideal subrings in Hopf algebroids}\label{sec:Hoids}

Let $\cH$ be a left Hopf algebroid such that $\sM{\cH} = {}_{{A\tensor{}\op{1}}}\cH$ is $A$-flat. Let $B$ be a right $\cH$-comodule $\op{A}$-subring of $\cH$ via $t$ for which 
\begin{equation}\label{eq:Bbeta}
(p \circ \beta^{-1})(\cH B^+ \tensor{A} \cH) = 0,
\end{equation}
where $p\colon \cH \tensor{\op{A}} \cH \to \cH \tensor{B} \cH$ is the canonical projection form \eqref{eq:Bxi}. For instance, this happens when there exists a morphism $\gamma'\colon B \to B \tensor{\op{A}} \cH$ such that $(\iota_B \tensor{\op{A}} \cH) \circ \gamma' = \gamma \circ \iota_B$.
Then, by the commutativity of \eqref{eq:Bxi}, the invertibility of the canonical map \eqref{Eq:betamap}
\[\beta\colon \Mt{\cH} \tensor{\op{A}} \tM{\cH} \to \Mto{\cH} \tensor{A} \sM{\cH}, \qquad x \tensor{\op{A}} y \mapsto \sum x_1 \tensor{A} x_2y,\]
entails that the morphism 
\[
\xi\colon {\cH_{\iota}} \tensor{B} {_\iota\cH} \to \Mto{\frac{\cH}{\cH B^+}} \tensor{A} \sM{\cH},
\]
from \Cref{prop:coidealsubalgebra}\ref{item:2.1b} is an isomorphism (see also \cite[Lemma 3.12]{AryanLaiachiPaoloJoost}). In this setting, we have the following remarkable result.

\begin{proposition}\label{prop:PsiPhiB}
Let $\cH$ be a left Hopf algebroid such that $\sM{\cH} = {}_{{A\tensor{}\op{1}}}\cH$ is $A$-flat. Let $B$ be a right $\cH$-comodule $\op{A}$-subring of $\cH$ via $t$ for which \eqref{eq:Bbeta} holds and such that $B \to \cH \rightrightarrows \cH \tensor{B} \cH$ is an equalizer. Then $B = \coinv{\cH}{\frac{\cH}{\cH B^+}}$.
\end{proposition}

\begin{proof}
It follows from inspecting \eqref{eq:PsiPhiB}, where now $\xi$ is an isomorphism and the first row can be completed to an equalizer diagram by adding $B$.
\end{proof}

In addition, we have the following counterpart.

\begin{proposition}\label{prop:PhiPsiI}
Let $\cH$ be a left Hopf algebroid such that $\sM{\cH} = {}_{{A\tensor{}\op{1}}}\cH$ is $A$-flat. Let $I$ be a left ideal coideal in $\cH$ for which $\cH \tensor{\op{A}} \coinv{\cH}{\frac{\cH}{I}}$ is still the equalizer of $\cH \tensor{\op{A}} \pi_I t \tensor{A} \cH$ and $\cH \tensor{\op{A}} (\pi_I \tensor{A} \cH) \circ \Delta$. 
Then, $\zeta \colon \cH \tensor{\op{A}} \coinv{\cH}{\frac{\cH}{I}} \to \cH \cotensor^{\frac{\cH}{I}} \cH$ from \Cref{prop:idealcoideal}\ref{item:2.2b} is an isomorphism. If, moreover, $\cH \cotensor^{\frac{\cH}{I}} \cH \rightrightarrows \cH \to \cH/I$ is a coequalizer, then $I = \cH \big(\coinv{\cH}{\frac{\cH}{I}}\big)^+$.
\end{proposition}

\begin{proof}
Under the additional hypotheses on $I$, the first conclusion follows from the commutativity of \eqref{eq:Izeta}, since now both lines are equalizers in the category of vector spaces and both vertical maps are isomorphisms. In such a case, if $\cH \cotensor^{\frac{\cH}{I}} \cH \rightrightarrows \cH \to \cH/I$ is also a coequalizer, then $\cH/I = \cH/\cH (\coinv{\cH}{\frac{\cH}{I}})^+$ because now both lines in \eqref{eq:PhiPsiI} are coequalizers and $\zeta$ is an isomorphism.
\end{proof}

Notice that if $\Mt{\cH} = \cH_{1 \tensor{} \op{A}}$ is $\op{A}$-flat or if $\coinv{\cH}{\frac{\cH}{I}} \subseteq \cH$ is left $\cH$-pure as an $\op{A}$-linear map (so that $\cH \tensor{\op{A}} \coinv{\cH}{\frac{\cH}{I}} \to \cH \tensor{\op{A}}\cH$ is still injective), then the first additional condition from \Cref{prop:PhiPsiI} is satisfied.

\begin{theorem}\label{thm:main}
Let $\cH$ be a Hopf algebroid such that $\sM{\cH} = {}_{{A\tensor{}\op{1}}}\cH$ is $A$-flat.
Then $\Phi$ and $\Psi$ induce a bijection:
\[
\xymatrix @R=0pt{
{\left\{ \begin{array}{c} \text{left ideal coideals } I \text{ for which} \\ \cH \cotensor^{\frac{\cH}{I}} \cH \rightrightarrows \cH \to \cH/I \text{ is a} \\ \text{coequalizer and } \coinv{\cH}{\frac{\cH}{I}} \subseteq \cH \text{ is} \\  \text{left and right }\cH\text{-pure} \\ \text{as } \op{A}\text{-bimodule} \end{array} \right\}} \ar@<+0.5ex>[r]^-{\Psi} & {\left\{ \begin{array}{c} \text{right } \cH\text{-comodule } \op{A}\text{-subrings } B \\ \text{via }t \text{ for which } B \to \cH \rightrightarrows \cH \tensor{B} \cH \\ \text{is an equalizer and } B \subseteq \cH \text{ is}  \\ \text{left and right }\cH\text{-pure as} \\ \op{A}\text{-bimodule and \eqref{eq:Bbeta} holds}\end{array} \right\}} \ar@<+0.5ex>[l]^-{\Phi} \\
I \ar@{|->}[r] & \coinv{\cH}{\frac{\cH}{I}} \\
\cH B^+ & B \ar@{|->}[l]
}
\]
\end{theorem}

\begin{proof}
Let $I$ be a left ideal coideal as in the statement. Since $\iota_I \colon \coinv{\cH}{\frac{\cH}{I}} \to \cH$ is $\cH$-pure as a morphism of right $\op{A}$-modules and since the following diagram, whose rows are now equalizers, commutes sequentially (one needs to use \cite[Proposition 3.7]{Schauenburg})
\[
\xymatrix @C=40pt{
\coinv{\cH}{\frac{\cH}{I}} \ar@{.>}[d]_-{\gamma'} \ar[r] & \cH \ar@<+0.5ex>[rr]^-{\pi_I t \, \tensor{A}\cH} \ar@<-0.5ex>[rr]_-{(\pi_I \tensor{A} \cH) \, \Delta} \ar[d]_-{\gamma} & & {\displaystyle \Mto{\frac{\cH}{I}} \tensor{A} \sM{\cH} } \ar[d]^-{\frac{\cH}{I} \tensor{A} \gamma } \\
\Mt{\coinv{\cH}{\frac{\cH}{I}}} \tensor{\op{A}} \tM{\cH} \ar[r] & \Mt{\cH} \tensor{\op{A}} \tM{\cH} \ar@<+0.5ex>[rr]^-{\pi_I t \, \tensor{A}\cH \tensor{\op{A}} \cH} \ar@<-0.5ex>[rr]_-{(\pi_I \tensor{A} \cH) \, \Delta \tensor{\op{A}} \cH} & & {\displaystyle \Mto{\frac{\cH}{I}} \tensor{A} \sMt{\cH} \tensor{\op{A}} \tM{\cH} }
}
\]
we may conclude that $\gamma\big(\coinv{\cH}{\frac{\cH}{I}}\big) \subseteq \coinv{\cH}{\frac{\cH}{I}} \tensor{\op{A}}\cH$ and so \eqref{eq:Bbeta} is satisfied for $\coinv{\cH}{\frac{\cH}{I}}$. Therefore, and in view of \Cref{prop:idealcoideal}\ref{item:2.2a}, $\Psi$ from left to right is well-defined.

Now, let $B$ be a right $\cH$-comodule $\op{A}$-subring as in the statement. By \Cref{prop:PsiPhiB}, $B = \coinv{\cH}{\frac{\cH}{\cH B^+}}$ and so, in view of \Cref{prop:coidealsubalgebra}\ref{item:2.1a}, $\Phi$ from right to left is well-defined.

To show that they are inverse bijections, observe that $\Psi\Phi(B) = \coinv{\cH}{\frac{\cH}{\cH B^+}} = B$ by \Cref{prop:PsiPhiB} and $\Phi\Psi(I) = \cH \big(\coinv{\cH}{\frac{\cH}{I}}\big)^+ = I$ by \Cref{prop:PhiPsiI}.
\end{proof}

\section{The Hopf algebra case}\label{sec:HopfAlgs}

Let $H$ be a Hopf algebra over the field $\K$. Since $H$ is, in particular, a Hopf algebroid over $\K$, we may apply the results from the previous sections to this case. As we are now working over a field, any extension is left and right pure and moreover \eqref{eq:Bbeta} is satisfied for any coideal subalgebra of a Hopf algebra. Therefore, \Cref{thm:main} rephrases as follows.

\begin{theorem}\label{thm:Hopfcase}
Let $H$ be a Hopf algebra over the field $\K$.
Then, the following is a bijection:
\begin{equation}\label{eq:Hopfcase}
\begin{gathered}
\xymatrix @R=0pt{
{\left\{ \begin{array}{c} \text{left ideal coideals} \\ \text{in } H \text{ for which} \\ H \cotensor^{\frac{H}{I}} H \rightrightarrows H \to H/I \\ \text{is a coequalizer} \end{array} \right\}} \ar@<+0.5ex>[r]^-{\Psi} & {\left\{ \begin{array}{c} \text{right coideal subalgebras} \\ \text{of } H \text{ such that} \\ B \to H \rightrightarrows H \tensor{B} H \\ \text{is an equalizer}\end{array} \right\}} \ar@<+0.5ex>[l]^-{\Phi} \\
I \ar@{|->}[r] & \coinv{H}{\frac{H}{I}} \\
H B^+ & B \ar@{|->}[l]
}
\end{gathered}
\end{equation}
\end{theorem}

\Cref{thm:Hopfcase} extends a number of well-known cases from the literature:
\begin{enumerate}[label=(\alph*),ref=(\alph*)]
	\item In the celebrated \cite[Theorem 4.3]{Takeuchi:1972}, Takeuchi proved that the assignments $\Phi$ and $\Psi$ from \eqref{eq:PhiPsi} restrict to a bijective correspondence between normal Hopf ideals and sub-Hopf algebras of a commutative Hopf algebra. By \cite[Theorem 3.1]{Takeuchi:1972}, a commutative Hopf algebra $H$ is faithfully flat over any sub-Hopf algebra $B$ and hence, in view of \cite[13.1 Theorem]{Waterhouse}, $B$ belongs to the right-hand side of \eqref{eq:Hopfcase}. Moreover, in view of \cite[\S4]{Takeuchi:1972}, any normal Hopf ideal $I$ in a commutative Hopf algebra $H$ is of the form $HB^+$ for a certain sub-Hopf algebra $B$ and hence, by \Cref{prop:coidealsubalgebra}\ref{item:2.1a}, $I$ belongs to the left-hand side of \eqref{eq:Hopfcase}.
	
	\item In \cite{Newman}, Newman showed that $\Phi$ and $\Psi$ induce a bijection between left ideal coideals and sub-Hopf algebras of a cocommutative Hopf algebra $H$. Since \cite[Theorem 3.1]{Takeuchi:1972} also states that a cocommutative Hopf algebra $H$ is faithfully flat over any sub-Hopf algebra $B$, as before we have that $B$ belongs to the right-hand side of \eqref{eq:Hopfcase} by \cite[13.1 Theorem]{Waterhouse}, again. By \cite[Corollary 3.4]{Newman}, any left ideal coideal $I$ of a cocommutative Hopf algebra $H$ is of the form $HB^+$ for a certain sub-Hopf algebra $B$ and so, by \Cref{prop:coidealsubalgebra}\ref{item:2.1a}, $I$ belongs to the left-hand side of \eqref{eq:Hopfcase}.
	
	\item In \cite{Takeuchi-Relative}, Takeuchi proved that $\Phi$ and $\Psi$ restrict to bijective correspondences between:
	\begin{enumerate}[label=(\roman*),ref=(\roman*)]
	\item\label{item:3a} Right coideal subalgebras $B \subseteq H$ of a commutative Hopf algebra $H$ over which $H$ is faithfully flat and Hopf ideals $I \subseteq H$ such that $H$ is faithfully coflat as left (equivalently, right) $H/I$-comodule.	
	In this case as well, $H$ faithfully flat over $B$ entails that $B $ belongs to the right-hand side of \eqref{eq:Hopfcase}. Analogously, $H$ faithfully coflat as left $H/I$-comodule entails that $I$ belongs to the left-hand side of \eqref{eq:Hopfcase}. The last claim can be proved as one does for the faithfully flat one: the diagram of vector spaces 
	\[
	\xymatrix @C=25pt{
	H \cotensor^{\frac{H}{I}} H \cotensor^{\frac{H}{I}} H \ar@<+1ex>[rr]^-{H \tensor{} \varepsilon \tensor{} {H}} \ar@<-1ex>[rr]_-{\varepsilon \tensor{} H \tensor{} H} && H \cotensor^{\frac{H}{I}} H \ar@<+0.5ex>[r]^-{\pi_I \tensor{} H} \ar[ll]|-{H \tensor{} \Delta} & {\displaystyle \frac{H}{I} \cotensor^{\frac{H}{I}} H \cong H } \ar@<+0.5ex>[l]^-{\Delta_H}
	}
	\]
	is a split coequalizer and hence, by left faithful coflatness of $H$ over $H/I$, $H \cotensor^{\frac{H}{I}} H \rightrightarrows H \to H/I$ is a coequalizer, too.
	
	\item Hopf subalgebras $B \subseteq H$ of a cocommutative Hopf algebra $H$ and left ideal coideals $I \subseteq H$ such that $H$ is faithfully coflat as $H/I$-comodule. The situation is the same as for point \ref{item:3a}.
	\end{enumerate}
	
	\item In \cite[Theorem 1.3]{Masuoka}, Masuoka improved Newman's result by showing that if the coradical of a Hopf algebra $H$ is cocommutative, then $H$ is faithfully flat as a module over any coideal subalgebra whose coradical is preserved by the antipode, and $H$ is faithfully coflat as a comodule over any quotient left module coalgebra. Thus, one might deduce a $1-1$ correspondence between right coideal subalgebras $K$ such that $S(K_0) = K_0$ and left ideal coideals. By the usual trick of the split (co)equalizer, faithful flatness guarantees that the coideal subalgebras considered by Masuoka lie in the right-hand side set of \eqref{eq:Hopfcase} and faithful coflatness that the left ideal coideals lie in the left-hand side one.
	
	\item In \cite[Theorem 1.4]{Schneider-remarks}, Schneider takes advantage of \cite{Takeuchi-Relative} to show that $\Phi$ and $\Psi$ restrict to a bijective correspondence between normal Hopf subalgebras $B \subseteq H$ of a Hopf algebra $H$ such that $H$ is right faithfully flat over $B$ and normal Hopf ideals $I \subseteq H$ such that $H$ is right faithfully coflat over $H/I$. As above, faithful (co)flatness entails that the normal Hopf subalgebras $B$ belong to the right-hand side of \eqref{eq:Hopfcase} and the normal Hopf ideals $I$ to its left-hand side.
\end{enumerate}

\begin{remark}\label{rem:extending}
\Cref{thm:Hopfcase} concretely extends the previous cases, as there exist coideal subalgebras $B \subseteq H$ of a (even commutative) Hopf algebra $H$ which satisfy the condition of \Cref{thm:Hopfcase} but over which $H$ is not faithfully flat. For instance, $\C[X,Y] \subseteq \cO(\textrm{SL}_2(\C))$, where the inclusion corresponds to the projection on the first row (the details can be found in the forthcoming \Cref{ex:extending}).

At the same time, the additional conditions involving (co)equalizers in \Cref{thm:main} and \Cref{thm:Hopfcase} are necessary conditions, as \Cref{prop:coidealsubalgebra} and \Cref{prop:idealcoideal} show, and cannot be avoided. For example, the subbialgebra $B \coloneqq\C[X]$ of the Hopf algebra $H \coloneqq \C[X^{\pm 1}]$ of Laurent polynomials is strictly contained in the equalizer of $H \rightrightarrows H \tensor{B} H$ (because all the powers $X^z$ for $z \in \ZZ$ are therein as well). In this case, in fact, $B^+ = B(X-1)$ and so $HB^+ = H(X-1) = H^+ = HH^+$. See also \cite[Example 3.18]{AryanLaiachiPaoloJoost} for a Hopf algebroid version of this example.
\end{remark}

\begin{example}\label{ex:extending}
Consider $B \coloneqq \C[X,Y]$ and
\[H \coloneqq \cO(\textrm{SL}_2(\C)) = \C[A,B,C,D]/\langle AD - BC - 1\rangle = \C[a,b,c,d \mid ad-bc=1].\]
The former becomes a right coideal subalgebra of the latter via 
\[\varrho \colon \C[X,Y] \to \C[a,b,c,d \mid ad-bc=1], \qquad \begin{cases} X \mapsto a \\ Y \mapsto b \end{cases},\]
and with respect to
\begin{gather*}
\Delta
\begin{pmatrix}
a & b \\ c & d 
\end{pmatrix}
= \begin{pmatrix}
a & b \\ c & d 
\end{pmatrix}
\otimes
\begin{pmatrix}
a & b \\ c & d 
\end{pmatrix} 
\qquad \text{and} \qquad
\delta\begin{pmatrix}
X & Y 
\end{pmatrix}
= \begin{pmatrix}
X & Y
\end{pmatrix}
\otimes
\begin{pmatrix}
a & b \\ c & d 
\end{pmatrix}.
\end{gather*}
This is an example of a flat (see \cite[3.4. Theorem]{Masuoka-Wigner}) but not faithfully flat extension: if it was faithfully flat, then the induced morphism $\calg_\C\big(\cO(\textrm{SL}_2(\C)),\C\big) \to \calg_\C\big(\C[X,Y],\C\big)$ would have been surjective, but the map
\[
\begin{array}{ccccccc}
\mathrm{SL}_2(\C) & \cong & \calg_\C\big(\cO(\mathrm{SL}_2(\C)),\C\big) & \to & \calg_\C\big(\C[X,Y],\C\big) & \cong & \C^2 \\
A & \mapsto & \mathrm{ev}_{A} & \mapsto & \mathrm{ev}_{A} \circ \varrho & \mapsto & (a_{11},a_{12})
\end{array}
\]
cannot be surjective as $(0,0)$ is not in the image.

Let us check it satisfies the condition of \Cref{thm:Hopfcase}, that is, that $B$ is the equalizer of $H \rightrightarrows H \tensor{B} H$. Since we are in the commutative setting, $H \tensor{B} H$ is the pushout of $H \supseteq B \subseteq H$ and so it is (isomorphic to) $K \coloneqq \C[s,t,u,v,w,z\mid sv - tu = 1 = sz - tw]$ with 
\[
\eta_1 \colon H \to K, \qquad \begin{cases} a \mapsto s \\ b \mapsto t \\ c \mapsto u \\ d \mapsto v \end{cases} \qquad \text{and} \qquad \eta_2 \colon H \to K, \qquad \begin{cases} a \mapsto s \\ b \mapsto t \\ c \mapsto w \\ d \mapsto z \end{cases},
\]
corresponding to $h \mapsto h \tensor{B}1$ and $h \mapsto 1 \tensor{B} h$, respectively. Now, a linear basis for $H$ is provided by the set
\[\big\{a^ib^jc^k, b^mc^nd^l \mid i,j,k,m,n\geq 0, l\geq 1\big\}\]
Take
\[h \coloneqq \sum_{i=0}^r\sum_{j=0}^e\sum_{k=0}^f\lambda_{ijk}a^ib^jc^k + \sum_{m=0}^p\sum_{n=0}^q\sum_{l=1}^o\mu_{mnl}b^mc^nd^l \in H \qquad \text{such that} \qquad h \tensor{B} 1 = 1 \tensor{B} h.\]
That is,
\[\sum_{i=0}^r\sum_{j=0}^e\sum_{k=0}^f\lambda_{ijk}s^it^ju^k + \sum_{m=0}^p\sum_{n=0}^q\sum_{l=1}^o\mu_{mnl}t^mu^nv^l = \sum_{i=0}^r\sum_{j=0}^e\sum_{k=0}^f\lambda_{ijk}s^it^jw^k + \sum_{m=0}^p\sum_{n=0}^q\sum_{l=1}^o\mu_{mnl}t^mw^nz^l\]
in $K$. By Bergman's Diamond Lemma (see \cite{Bergman}), the elements
\begin{equation}\label{eq:basis}
\big\{s^it^ju^k,t^mu^nv^l,s^it^jw^k,t^mw^nz^l\mid k \geq 1, l\geq 1\big\}
\end{equation}
are linearly independent: in the terminology of \cite{Bergman}, a reduction system
\[S = \{\sigma = (W_\sigma,f_\sigma) \mid K = \C[s,t,u,v,w,z]/\langle W_\sigma - f_\sigma \mid\sigma \in S\rangle\}\] 
for which all the ambiguities are resolvable relative to the lexicographic ordering is given by 
\[\{(sv,tu+1), (sz,tw+1), (tuz,tvw + v - z)\}\] 
and the words in \eqref{eq:basis} are irreducible monomials with respect to this reduction system. Therefore,
\[\lambda_{ijk} = 0 \text{ for all } k\geq 1 \qquad \text{and} \qquad \mu_{mnl} = 0 \text{ for all } l\geq 1\]
and so 
\[h = \sum_{i=0}^r\sum_{j=0}^e\lambda_{ij0}a^ib^j \in B.\]
\end{example}

\end{document}